\documentclass[12pt,a4paper]{article}
\usepackage[british]{babel}  
\usepackage[numbers,sort&compress]{natbib}
\usepackage[a4paper,top=2cm,bottom=2cm,left=2.5cm,right=2.5cm,marginparwidth=1.75cm]{geometry}
\usepackage{natbib} 

\usepackage{amsthm}
\usepackage{amsmath}
\usepackage{graphicx}
\usepackage[colorlinks=true, allcolors=blue]{hyperref}
\usepackage{hyperref}
\usepackage{orcidlink}
\usepackage[title]{appendix}
\usepackage{mathrsfs}
\usepackage{amsfonts}
\usepackage{booktabs} 
\usepackage{caption}  
\usepackage{threeparttable} 
\usepackage{algorithm}
\usepackage{algorithmicx}
\usepackage{algpseudocode}
\usepackage{listings}
\usepackage{enumitem}
\usepackage{chngcntr}
\usepackage{booktabs}
\usepackage{lipsum}
\usepackage{subcaption}
\usepackage{authblk}
\usepackage[T1]{fontenc}    
\usepackage{csquotes}       
\usepackage{diagbox}

\newcommand{\ox}{\bar{x}}

\newcommand{\intset}{{\rm int}}
\newcommand{\cl}{{\rm cl}}
\newcommand{\bd}{{\rm bd}}
\newcommand{\dom}{{\rm dom}}
\newcommand{\gph}{{\rm gph}}

\newcommand{\epi}{{\rm epi}}

\newcommand{\N}{\mathbb{N}}
\newcommand{\R}{\mathbb{R}}

\newcommand{\la}{\langle}
\newcommand{\ra}{\rangle}

\usepackage{setspace}
\usepackage{titlesec}
\titleformat{\section} 
{\normalfont\Large\bfseries}{\thesection.}{1em}{}
  
\usepackage{float}   
\usepackage{caption} 
\newtheorem{theorem}{Theorem}[section]
\newtheorem{proposition}{Proposition}[section]

\newtheorem{exercise*}{Exercise}
\newtheorem{definition}{Definition}[section]
\newtheorem{remark}{Remark}[section]
\newtheorem{example}{Example}[section]

\makeatletter
{
\end{list}%
}
\makeatother

\title{$\ell_0$-Norm Multiobjective Optimization Models Motivated by Applications to Proton Therapy}
\author[1]{Xiaoda Cong}
\author[1]{Xuanfeng Ding}
\author[2]{Boris Mordukhovich}
\author[2]{Anh Vu Nguyen}
\author[3]{Lewei Zhao}
\affil[1]{\small Proton Therapy Center, Corewell William Beaumont Hospital, Royal Oak, MI 48073}
\affil[2]{\small Department of Mathematics, Wayne State University, Detroit, MI 48202}
\affil[3]{Department of Radiation Oncology, MedStar Georgetown University Hospital, Washington DC, 20007}
\date{}  

\begin{document}
\maketitle
\vspace*{-0.3in}
{\bf Dedicated to Professor Christiane Tammer in honor of her 70th birthday}

\begin{abstract}
\vspace*{-0.2in}
\end{abstract}
The paper is devoted to investigating single-objective and multiobjective optimization problems involving the $\ell_0$-norm function, which is nonconvex and nondifferentiable. Our motivation comes from proton beam therapy models in cancer research. The developed approach uses  subdifferential tools of variational analysis and the Gerstewitz (Tammer) scalarization function in multiobjective optimization. Based on this machinery, we propose several algorithms of the subgradient type and conduct their convergence analysis. The obtained results are illustrated by numerical examples, which reveal some characteristic features of the proposed algorithms and their interactions with the gradient descent.\\[1ex] 

\section{Introduction}\label{intro}

This paper revolves around variational analysis and optimization of single-objective and multiobjective models involving the so-called {\em $\ell_0$-norm function} $\|\cdot\|_0$, which signifies the number of nonzero components in a given vector. The function $\|\cdot\|_0$ (it is not actually a norm) is nondifferentiable at the origin being also nonconvex. Therefore, the standard machinery of classical and convex analysis is not applicable to the study of the $\ell_0$-norm function, which stimulate us to use for these purposes advanced tools of variational analysis and generalized differentiation dealing with nonsmoothness and nonconvexity. 

A profound interest to optimization problems involving the $\ell_0$-norm function has arisen in more recent years from applications to machine learning, signal processing, compressed sensing, etc.; see, e.g., \cite{khanh} and the references therein. It has been realized that such problems are {\em NP-hard} and reflect the {\em sparsity} in optimization, which is a challenging issue. To this end, we refer the reader to the primal-dual active set with continuation (PDASC) algorithm developed in \cite{jiao2015primal} for solving the $\ell_0$-regularized least-squares problem that frequently arises in compressed sensing. A practical method to solve $\ell_0$-norm problems for neural networks was developed in \cite{louizos2017learning}. In \cite{lee2022self}, $\ell_0$-sparse optimization was used in minimizing the number of features in deep learning. 

In recent years, the $\ell_0$-norm function has gained significant attention in the field of medical physics. Using the $\ell_0$-norm encourages sparsity and enables the extraction of meaningful information from the noisy one and/or reducing the unnecessary delivery, which plays a pivotal role to ensure accurate diagnostics and patient treatment efficiency. The study in \cite{lyu2019image} compared $\ell_p$-regularization for $p=0, 1/2, 2/3, 1$ in image-domain multimaterial decomposition for dual-energy computed tomography. It is shown therein that the smaller is $p$, the more nonconvex will be the problem, and thus it is more difficult to find optimal sparsity solutions. In particular, $\ell_0$-sparsity optimization techniques have been leveraged in radiation therapy. 

Proton therapy is an advanced type of radiation treatment for cancer that uses high-energy protons to precisely kill tumor, while minimizing damage to surrounding healthy tissues. In proton therapy, especially proton arc therapy, treatment delivery efficiency plays a crucial role in the routine clinical implementation. Previous preliminary investigations indicated that the number of energy layers and spot impact not only affects plan quality but also the treatment delivery time. Therefore, finding an optimal sparsity level of the energy layer or spot to ensure fast treatment delivery while maintaining clinically acceptable plan quality remains a significant challenge. In terms of the energy layer sparsity solution, Gu et al. \cite{Gu} formulated an optimization problem by integrating an ${\ell}_{1/2}$-regularization term for energy layer selection. Meanwhile, in the direction of spot number sparsity,
\cite{Zhao} introduce the $\ell_0$-norm concept to reduce the unnecessary spots in the complicated SPArc treatment plan. Recently, the alternating direction method of multipliers has been developed and compared with the PDASC method, which shows promising results reported in \cite{fan2024optimizing}.

Proton arc therapy delivers proton beams while the treatment machine rotates around the patient. A challenge to proton arc therapy is how to reduce beam delivery time. One approach is to decrease spot number. The alternating direction method of multipliers and the aforementioned PDASC were applied in \cite{Zhao2022ADMM}, \cite{zhao2023first} and \cite{fan2024optimizing} from the viewpoint of $\ell_0$-sparsity optimization in proton arc therapy to reduce the corresponding spot number. These strategies prioritize the principle of sparsity while assuming that optimal solutions are either nearly sparse, or can be made sparse through an appropriate transformation. 

Reducing spot number will reduce degree of freedom, which may degrade the plan quality. Thus minimizing beam delivery time and optimizing plan quality becomes two conflict goals, making the model a multicriterial optimization problem. Motivated by such practical multiobjective models arising in proton beam therapy that are unavoidably contain the $\ell_0$-norm function and the like, we aim here to provide some simplified descriptions of such model and then design novel {\em subgradient-based algorithms} to find their {\em efficient/Pareto optimal} solutions by using the precise calculation of the {\em basic/limiting subdifferential} of the $\ell_0$-norm function in the sense of Mordukhovich; see  \cite{mordukhovich2006,rockafellar1998} and the references therein. Our strategy is to start with a single-objective optimization problem of minimizing the sum of a convex differentiable function and the $\ell_0$-norm. Then we formulate a multiobjective optimization problem whose vector objective consists of scalar components of the above type. To deal with the latter problem, we employ the two scalarization techniques: the standard {\em weight-sum} approach and the (more general) {\em Gerstewitz scalarization} \cite{tammer1983,tammer2015}. In this way, the designed algorithm in the scalar $\ell_0$-norm case induces the corresponding algorithms for the multiobjective problem in question by using the aforementioned two scalarization techniques. A detailed convergence analysis is conducted fore all the proposed algorithms. Numerical calculations by using the code developed in the Proton Beam Therapy Group of the William Beaumont Hospital demonstrate the reliability and efficacy of our algorithms and thus signify a promising direction of our study to handle multifaceted optimization problems with potential wide-ranging impacts, particularly in the areas of cancer research and medical physics.

The rest of the paper is structured as follows. Section~\ref{sec:def} presents basic definitions of the normal cone, subdifferential, and $\ell_0$-norm functions with brief discussions of their underlying properties. In Section~\ref{sec:sub-cal}, we provide a complete calculation of the basic subdifferential of the $\ell_0$-norm function important for the subsequent algorithmic design. Section~\ref{sec:optim} provides the formulation of both single-objective and multiobjective optimization problems by using weight-sum scalarization. In Section~\ref{sec:tammer}, we explore some properties of the Gerstewitz scalarization function needed for our algorithmic developments, This allows us to formulate and investigate in Section~\ref{sec:G-l0} the scalarized version of $\ell_0$ multiobjective optimization problem involving the Gerstewitz function with an $\ell_0$ addition. Based on the above, the algorithms for $\ell_0$ scalar and multiobjective optimization are designed in Section~\ref{sec:algor}. Their detailed convergence analysis is conducted in Section~\ref{sec:conver}. In Section~\ref{sec:numer}, we implement the designed algorithms in numerical calculations for typical examples and discuss their characteristic features. Finally, Section~\ref{sec:conc} summarizes the main achievements of the paper and lists some topics of our future research.

\section{Basic Definitions and Discussions}\label{sec:def}

Using a geometric approach to variational analysis and generalized differentiation \cite{mordukhovich2018,mordukhovich2006}, we define first generalized normals to sets. Given a nonempty set $\Omega\subset\R^n$, the  (Fr\'echet) {\em regular normal cone} to $\Omega$ at $\bar{x} \in \Omega$ is given by
\begin{align}\label{cthuc1}
\widehat{N}(x;\Omega):= \left\{v \in\R^n\;\Big|\;\limsup\limits_{u \xrightarrow{\Omega} x} \dfrac{\langle v,u-x \rangle}{\lVert u-x\rVert} \leq 0\right\},
\end{align}
where $u\xrightarrow{\Omega} x$ means that $u \rightarrow x $ with $u \in \Omega$ . If  $x \notin \Omega$, we put $\hat{N}(x,\Omega):= \emptyset$. The (Mordukhovich) {\em basic/limiting normal cone} $N(\bar{x};\Omega)$ to $\Omega$ at $\bar{x}\in\Omega$ is defined by
\begin{equation}\label{cthuc2}
\begin{array}{ll}
 N(\bar{x};\Omega):=\Big\{v\in\R^n\;\Big|&\exists\,x_k\xrightarrow{\Omega}\bar x,\;\exists\,v_k\to v\;\mbox{ as }\;k\to\infty\\
 &\mbox{such that } \;v_k\in\widehat N(x_k;\Omega)\;\mbox{ for all }\;k\in\mathbb N:=\{1,2,\}\Big\}
 \end{array}
\end{equation}
with $ N(\bar{x};\Omega):=\emptyset$ whenever $\ox\notin\Omega$. Both normals cones in \eqref{cthuc1} and \eqref{cthuc2} reduce to the classical normal cone of convex analysis
\begin{equation*}
N(\bar x;\Omega)=\big\{v\in\R^n\;\big|\;\la v,x-\bar{x}\ra\le 0\;\mbox{ for all }\;x\in\Omega\big\}
\end{equation*}
if the set $\Omega$ is convex. If $\Omega$ is a nonconvex set, then properties of the limiting normal cone \eqref{cthuc2} are much better than those for its regular counterpart \eqref{cthuc1}, which may be even trivial (i.e., $\widehat{N}(\bar x;\Omega)=\{0\}$) while $\bar x$ is a boundary point as for the set $\Omega:=\{(x_1,x_2)\in\R^2\;|\;x_2\ge-|x_1|\}$ at the origin $\bar x=(\bar x_1,\bar x_2)=(0,0)$. 

Consider next an extended real-valued function $\varphi :\R^n\rightarrow \overline{\mathbb{R}} := (-\infty, \infty ]$ and the associated {\em domain} and {\em epigraph} sets of $\varphi$ given by
$$ 
\dom\, \varphi:=\big\{x \in\R^n\;\big|\;\varphi(x) < \infty\big\}\;\mbox{ and }\;\epi\; \varphi:=\big\{(x,\mu) \in\R^{n+1}\;\big|\;\mu\ge \varphi(x)\big\},
$$
respectively. We define the (Mordukhovich) {\em basic/limiting subdifferential} of $\varphi$ geometrically via the basic normal cone \eqref{cthuc2} to the epigraph by
\begin{align}\label{sub}
\partial \varphi(\bar{x}):=\big\{v\in\R^n\;\big|\;(v,-1) \in N((\bar{x}, \varphi(\bar{x})); \epi \, \varphi\big)\}
\end{align}
if $\bar x\in\dom\,\varphi$ and $\partial\varphi(\bar x):=\emptyset$ otherwise.  There are various analytic representations of \eqref{sub}, which can be found in the books \cite{mordukhovich2018,mordukhovich2006,rockafellar1998} and the references therein. These books, as well as the quite recent one \cite{mordukhovich2024}, contain comprehensive calculus rules and other results for the subdifferential \eqref{sub} and related constructions with a broad variety of applications that are mainly based  on the variational/extremal principles of variational analysis. Note that the subgradient set \eqref{sub} reduces to the classical gradient $\{\nabla\varphi(\bar x)\}$ for smooth functions and to the subdifferential of convex analysis if $\varphi$ is convex.\vspace*{0.03in}

The primary  object of our study and applications in this paper is the following real-valued function defined on finite-dimensional spaces.

\begin{definition}\label{def:l0}
The {\sc $\ell_0$-norm function} $\Vert \cdot \Vert_0: \R^n \rightarrow \R$ is given by
\begin{equation}\label{l0}
\Vert x \Vert_0:=\mbox{number of nonzero components of }\;x.
\end{equation}
\end{definition}

This function is clearly nondifferentiable and nonconvex, and we intend to utilize its basic subdifferential \eqref{sub} for its study and applications to optimization model. To begin with, let us illustrate the calculation of \eqref{l0} in the three-dimensional space. 

\begin{example} For the $\ell_0$-norm function \eqref{l0} on $\R^3$, we have
\begin{align*}
&\Vert (0,-1,4) \Vert_0=2,\\
&\Vert (1,0,0) \Vert_0 =1,\\
&\Vert (0,0,0) \Vert_0 =0,\\
&\Vert (1,2,3) \Vert_0 =3. 
\end{align*}
\end{example}

To further illustrate \eqref{l0}, we interpret it as a piecewise function (say, in $\R^2$) defined by
\begin{equation} \label{f-l0}
f(x,y):=\begin{cases}
0 \quad \text{if} \quad x=y=0,\\
1 \quad \text{if either} \quad x \neq 0,\;y=0, \quad \text{or} \quad \text{if}\quad x=0,\; y\neq 0,\\
2 \quad \text{otherwise}.
\end{cases}
\end{equation}
Based on \eqref{f-l0}, the graph of the $\ell_0$-norm function on $\R^2$ is 
$$
\gph \Vert\cdot \Vert_0 =\big\{(0,0,0)\big\} \bigcup\big\{(x,0,1) \mid x \neq 0\big\} \bigcup\big\{(0,y,1) \mid y \neq 0\big\} \bigcup\big\{(x,y,2) \mid x, y\neq 0\big\},
$$
which is depicted in Figure~\ref{fig:L_0 graph}.
\begin{figure}[H]
\centering
\includegraphics[width=1\textwidth]{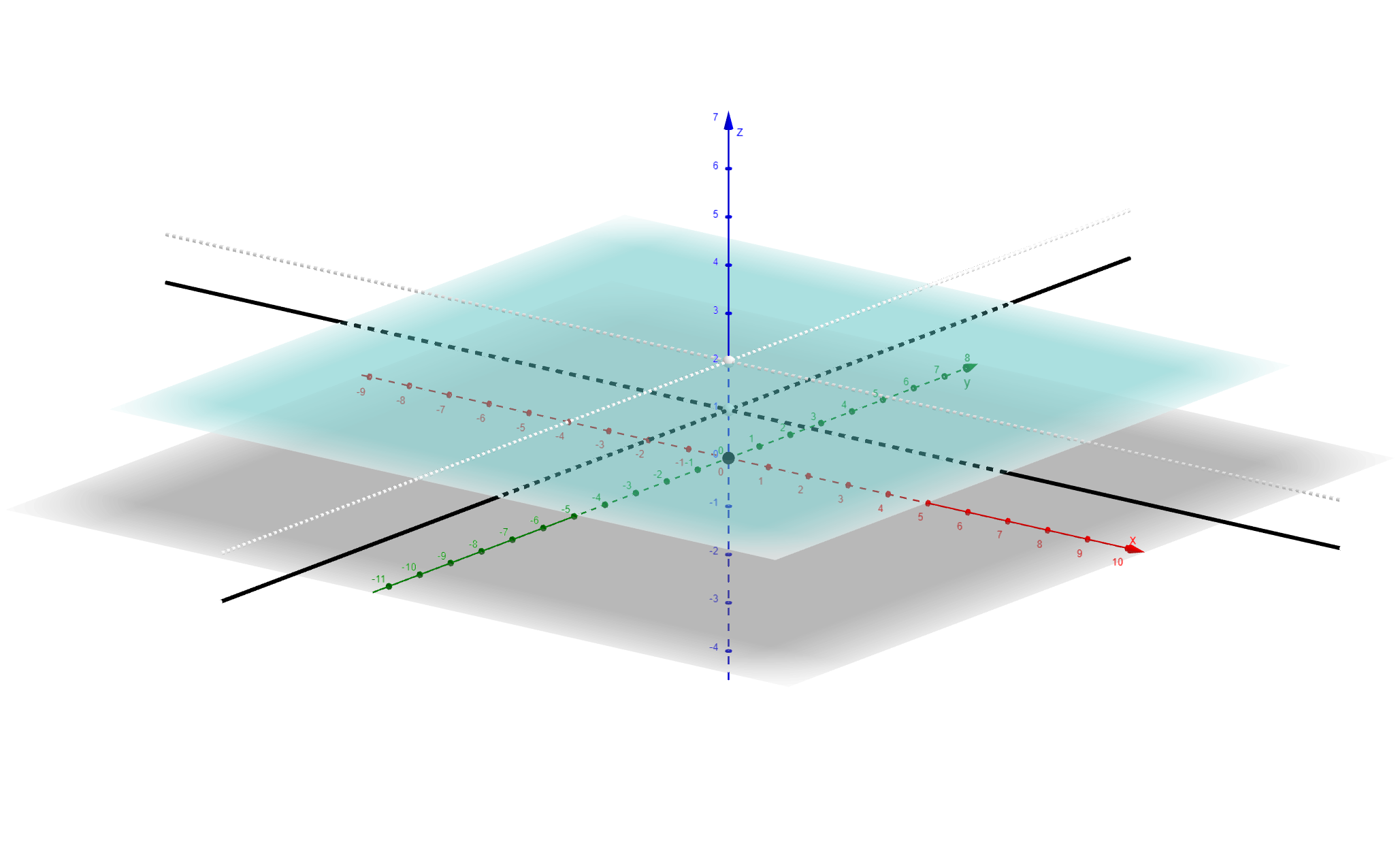}
\caption{Graph of the $\ell_0$-norm function}
\label{fig:L_0 graph}
\end{figure}

\section{Subgradient Calculation for the $\ell_0$-Norm Function}\label{sec:sub-cal}

In this section, we completely calculate the basic subdifferential \eqref{sub} of the $\ell_0$-norm function \eqref{l0} at any point of an arbitrary finite-dimensional space. 

\begin{theorem}\label{sub-cal} The subdifferential of $\Vert \cdot \Vert_0$ at any $\bar{x} = (\bar{x}_1, \ldots, \bar{x}_n) \in \R^n$ is calculated by
\begin{align*}
&\partial\Vert \cdot \Vert_0 (\bar{x}) =\big\{v=(v_1, \ldots, v_n) \in \R^n\big\}, \\
&\text{where }\;v_i\;\text{ are defined as }\;\begin{cases}
v_i =0 \quad \text{if} \quad \bar{x}_i \neq 0,\\
v_i \in \R \quad \text{if} \quad \bar{x}_i =0.
\end{cases}
\end{align*}
\end{theorem}
\begin{proof} For simplicity, we verify the claimed formula for the function $\|\cdot\|_0$ defined on $\R^2$. Based on the definitions in \eqref{sub} and \eqref{l0}, split the proof into the following cases:\\[0.5ex]
{\bf Case~1}: {\em $\bar{x} =(\bar{x}_1,\bar{x}_2)$ with $\bar{x}_1,\bar{x}_2 \neq 0$}. Taking a sequence $x_{1,k} x_{2,k} \rightarrow \bar{x}_1,\bar{x}_2$, we get that $x_{1,k}, x_{2,k} \neq 0$ when $x_{1,k}, x_{2,k}$ are sufficiently close to $\bar{x}_1, \bar{x}_2$. This tells us by \eqref{l0} that $\Vert (x_{1,k}, x_{2,k}) \Vert_0 =2$. Picking now  a sequence $((x_{1,k}, x_{2,k},2) \rightarrow (\bar{x}_1, \bar{x}_2,2)$ in $\epi \Vert \cdot \Vert_0$, we deduce from the definitions of the subdifferential \eqref{sub} and the basic normal cone \eqref{cthuc2} that
\begin{align*}
\limsup_{(x_{1,k}, x_{2,k}) \rightarrow (\bar{x}_1, \bar{x}_2)}\dfrac{\langle(v_1,v_2,-1),(x_{1,k}, x_{2,k},2) - (\bar{x}_1, \bar{x}_2,2)\rangle}{\Vert (x_{1,k}, x_{2,k}) - (\bar{x}_1, \bar{x}_2)\Vert} \leq 0,
\end{align*}
which is equivalently written as     
\begin{align*}
\limsup_{(x_{1,k}, x_{2,k}) \rightarrow (\bar{x}_1, \bar{x}_2)}\dfrac{\langle(v_1,v_2),(x_{1,k}, x_{2,k}) - (\bar{x}_1, \bar{x}_2)\rangle}{\Vert (x_{1,k}, x_{2,k}) - (\bar{x}_1, \bar{x}_2)\Vert} \leq 0.
\end{align*}
Choosing $x_{i,k}= \bar{x}_i$ for $i=1,2$ gives us 
$$
\limsup_{x_{j,k} \rightarrow\bar{x}_j}\dfrac{v_j(x_{j,k} -\bar{x}_j)}{\Vert x_{j,k} -\bar{x}_j \Vert } \leq 0\;\text{ for }\;j\neq i,
$$ 
which readily implies that $v_j =0$, and thus $v=(0,0)$ as claimed.\\[0.5ex]
{\bf Case~2}: {\em $\bar{x} =(0,\bar{x}_2)$ with $\bar{x}_2\neq 0$}. Choosing $x_{1,k}=0$ allows us to conclude, similarly to the proof in Case~1, that $v_2 =0$. Now let $x_{2,k} = \bar{x}_2$, and let $\{x_{1,k}\}$ be a a sequence on nonzero numbers converging to $0$. Then we have $\Vert(x_{1,k}, x_{2,k})\Vert_0 =2$ and
\begin{align*}
\limsup_{(x_{1,k}, x_{2,k}) \rightarrow (\bar{x}_1, \bar{x}_2)}\dfrac{\langle(v_1,v_2,-1),(x_{1,k}, x_{2,k},2) - (\bar{x}_1, \bar{x}_2,2)\rangle}{\Vert (x_{1,k}, x_{2,k}) - (\bar{x}_1, \bar{x}_2)\Vert} \leq 0,
\end{align*}
which can be equivalently rewritten as    
\begin{align*}
\limsup_{(x_{1,k}, x_{2,k}) \rightarrow (\bar{x}_1, \bar{x}_2)}\dfrac{\langle(v_1,v_2,-1),(x_{1,k}, x_{2,k},2) - (\bar{x}_1, \bar{x}_2,1)\rangle}{\Vert (x_{1,k}, x_{2,k}) - (\bar{x}_1, \bar{x}_2)\Vert} \leq 0
\end{align*}
and brings us therefore to the inequality
\begin{align*}
\limsup_{x_{1,k} \rightarrow \bar{x}_1}\dfrac{v_1(x_{1,k} -\bar{x}_1) -1}{\Vert x_{1,k} - \bar{x}_1\Vert} \leq 0.
\end{align*}
For any $v_1$, we choose $k$ to be so large that $v_1(x_{1,k}- \bar{x}_1) -1 \leq 0$, which ensures that $\{v=(v_1,v_2)\} = \{(v_1,0) \mid v_1 \in \R\}$
as claimed in this case.\\[0.5ex]
{\bf Case~3}: $\bar{x}=(\bar{x}_1,0)$.  The calculation of $\partial\Vert \cdot \Vert_0 (\bar{x})$ is similar to the proof in Case~2.\\[0.5ex]
{\bf Case~4}: $\bar{x}=(0,0)$. The proof is similar to Case~2 when we fix $x_{i,k}=0$ for either $i=1$ or $i=2$. This completes the proof of the theorem. 
\end{proof}

\section{Scalar and Vector Problems of $\ell_0$ Optimization}\label{sec:optim}

First we formulate here the following single-objective problem of optimization involving the $\ell_0$-norm function. Let $f: \R^n \rightarrow \R$ be a smooth and convex function whose gradient is Lipschitz continuous with a constant $L$. Consider the {\em $\ell_0$ optimization problem}:
$$ 
\text{($\ell_0$NOP)}  \qquad \min_{x \in \R^n} f(x)+\Vert x \Vert_0.
$$
The objective function in \text{($\ell_0$NOP)} is nonconvex and nondifferentiable. Let us illustrate its behavior by the following example.

\begin{example}\label{exa-loc} {\rm Consider the quadratic function $f: \R^2 \rightarrow \R$ defined by $f(x,y): = (x-\dfrac{1}{2})^2+\dfrac{1}{4}(y+1)^2$. The graph of $f(x,y) +\Vert (x,y) \Vert_0$ is depicted in Figure~\ref{fig:L_0 +f graph}.}
\end{example}
\begin{figure}[H]
\centering
\includegraphics[width=0.7\textwidth]{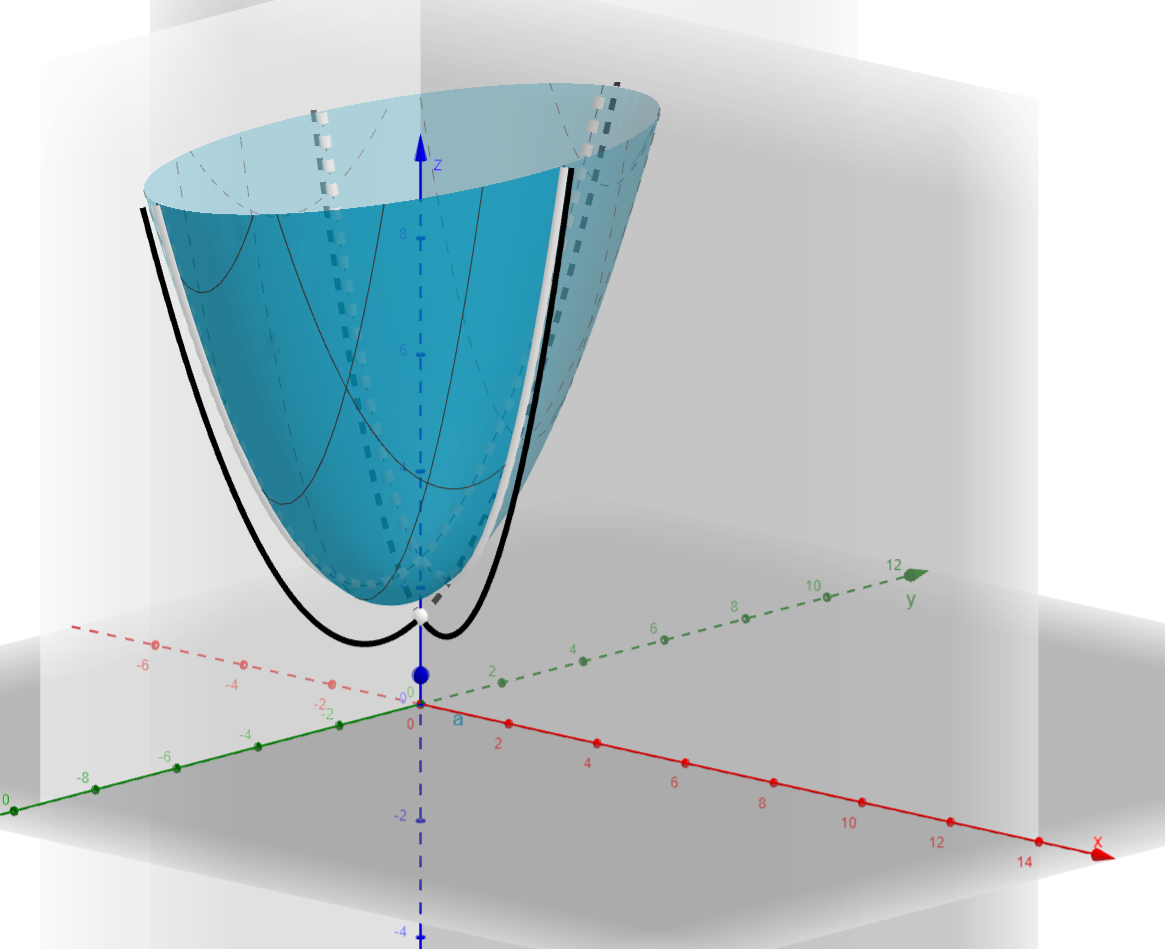}
\caption{Graph of $f(\cdot)+ \Vert \cdot \Vert_0$}
\label{fig:L_0 +f graph}
\end{figure}

As seen from Figure~\ref{fig:L_0 +f graph}, the objective function in Example~\ref{exa-loc} has  multiply local minimizers. Our goal concerning problem \text{($\ell_0$NOP)} is design an efficient subgradient algorithm, which allows us to find at least one {\em local optimal solution} to this problem. To proceed in this direction, we first construct, based on the subdifferential calculation in Theorem~\ref{sub-cal}, a certain {\em projection matrix} associated with $x\in\R^n$ over the hyperplane where some of the components of $x$ are zero. Define the {\em projection mapping} $\Tilde{I}_i : \R^n \rightarrow \R^n$ as follow: for $x_i=(x_{1,i}, \ldots , x_{n, i}) \in \R^n$, let
$\Tilde{I}_i$ be an $n \times n$ matrix such that 
\begin{align*}
&a_{j,k} = 0 \quad \text{if } j \neq k,\\
&a_{j,j} = 0 \quad \text{if } x_{j,i} =0,\\
&a_{j,j} = 1 \quad \text{if } x_{j,i} \neq 0.
\end{align*}  

The next example illustrates the construction of the projection matrix in $\R^3$.

\begin{example} 
For $x_i=(0,2,-3)$, we have
\begin{align*}
\Tilde{I}_i = \begin{bmatrix}
0 &0 &0 \\
0 &1 &0\\
0 &0 &1
\end{bmatrix}.
\end{align*}
\end{example}

Observe that the projection mapping $\Tilde{I}_i$ is defined by the value of $x_i$ with $\Tilde{I}_i x_i =x_i$. For each projection mapping, we have the corresponding {\em projection hyperplane} $I_i:= \Tilde{I}_i (x_1, \ldots, x_n)$.  This does not make any change in the value of $x_i$ while helping us to determine the direction for each iteration of our algorithms designed in Section~\ref{sec:algor}. 

The following important result allows us to find a local minimizer of problem \text{($\ell_0$NOP)} by minimizing the smooth and convex function $f$ therein over the corresponding projection hyperplane generated by the subgradient calculation for the $\ell_0$-norm function.

\begin{theorem}\label{theorem of local min}
If $\bar{x}$ solves the {\sc projection minimization problem} 
\begin{align*}
{\rm(PP)} \qquad \min_{x \in I_i} \quad f(x),
\end{align*}
then $\bar{x}$ is a local minimizer of the original $\ell_0$ optimization problem {\rm($\ell_0$OP)}.
\end{theorem}
\begin{proof}
Picking $x \in I_i$ and taking into account that $\bar{x}$ minimizes $f$ over the hyperplane $I_i$ ensures that $\Vert \bar{x}\Vert_0 =\Vert x \Vert_0$ implying therefore that $f(\bar{x}) +\Vert \bar{x} \Vert_0 \leq f(x) +\Vert x \Vert_0$. If $x \notin I_i$ being sufficiently close to $\bar{x}$, then it follows for each component $\bar{x}_j \neq 0$ of $\bar x$ that $x_i \neq 0$. In the case where $\bar{x}_i=0$, there clearly exist some nonzero components of $x_i$. This yields
\begin{align*}
\Vert \bar{x} \Vert_0 \leq \Vert x \Vert_0 -1.
\end{align*}
By the continuity of $f$, for any $\epsilon>0$ there is $\delta >0$ such that 
\begin{align*}
\vert f(x) - f(\bar{x}) \vert < \epsilon\;\mbox{ whenever }\;\Vert x-\bar{x} \Vert < \delta,
\end{align*}
and therefore $f(\bar{x}) < f(x) +\epsilon$. In this way, we arrive at the estimates
\begin{align*}
f(\bar{x}) +\Vert \bar{x} \Vert_0 &\leq f(x) +\Vert x\Vert_0 +\epsilon -1 \\
&\leq f(x) +\Vert x\Vert_0
\end{align*}
and thus complete the proof of the theorem.
\end{proof}

Now we formulate the $\ell_0$ multiobjective optimization problem, which is a multiobjective/vector version of the scalar $\ell_0$ optimization problem ($\ell_0$NOP). Consider the vector mapping $F: \R^n \rightarrow \R^m$ with $F(x):=\{F_1 (x), \ldots, F_m (x)\}$, where each $F_i(x):= f_i(x) +\Vert x\Vert_0$ as $i=1,\ldots,m$ is such that $f_i$  are convex differentiable functions with Lipschitz continuous gradients. Denote by $L_i$ the Lipschitz constant of the corresponding gradient and set $L:=\max\{L_i\mid i=1,\ldots,m\}$. The {\em $\ell_0$ multiobjective  optimization problem} is defined by
$$
(\text{MO-$\ell_0$NP}) \qquad \min_{x \in \R^n} F(x)=(F_1(x), \ldots, F_m(x)),
$$
where 'min' is understood in the sense of {\em Pareto optimality/efficiency}. 

\begin{definition}
A vector $\bar x\in\R^n$ is called a {\sc local Pareto optimal solution} to the multiobjective optimization problem {\rm(\text{MO-$\ell_0$NP})} if there exists no other point $x$ in a neighborhood $U$ of $\bar x$ for which we have
$$ 
F_i (x) \leq F_i (x),\quad i=1,\ldots,m.
$$
\end{definition}

A natural way to study multiobjective optimization problems is to use a certain {\em scalarization}, which converts the multiobjective problem in question to some problem of scalar (single-objective) optimization. the most simple scalarization approach is employing the {\em weight-sum scalarization}. For problem (\text{MO-$\ell_0$NP}), such a scalarization looks as follows:
\begin{equation}\label{weight}
\min_{x\in \R^n}\qquad\text{weight sum} (x): =\sum_{i=1}^m \omega_i F_i(x)
\end{equation}
with the weight coefficients satisfying $\omega_i \geq 0$ and $\sum_{i=1}^m \omega_i=1$. It is easy to observe the following relation between optimal solutions of the multiobjective and scalarized problems.

\begin{proposition}\label{Pareto solution}
A local minimizer of the weight sum function in \eqref{weight} is a local Pareto optimal solution to the multiobjective optimization problem {\rm(\text{MO-$\ell_0$NP})}. 
 \end{proposition}
\begin{proof}
Pick an arbitrary local minimizer $\bar x$ in \eqref{weight} and show that it is a local Pareto solution to the multiobjective problem (\text{MO-$\ell_0$NP}). Assuming the contrary, we find a vector $x$ in some neighborhood $U$ of $\bar x$ satisfying the conditions
\begin{align*}
&F_i(x) \leq F_i(\bar x) \quad \text{for all} \quad i=1,\ldots,m,\\
&F_i(x) < F_i(\bar x) \quad \text{for some} \quad i.
\end{align*}
This readily implies, by using the properties of the weight coefficients in \eqref{weight} that
$$ 
\text{weight sum}(x) < \text{weight sum}(\bar x).
$$
The latter contradicts the local optimality of $\bar x$ in \eqref{weight} and thus completes the proof. 
\end{proof}

Using the weight condition $\sum_{i=1}^m \omega_i =1$ and the form of the components $F_i$ in  (\text{MO-$\ell_0$NP}), the weight sum function in \eqref{weight} can be written as
$$
\text{weight sum} (x) = \sum_{i=1}^m \omega_i f_i(x) + \Vert x\Vert_0,
$$
which shows that \eqref{weight} is a problem of scalar $\ell_0$ optimization of type (\text{$\ell_0$NOP}).

\section{Gerstewitz Scalarization Function}\label{sec:tammer}

In this section, we start exploring another approach to scalarize the $\ell_0$ multiobjective optimization problem (\text{MO-$\ell_0$NP}) that is based on using the {\em Gerstewitz scalarization function} introduced in \cite{tammer1983} and then developed in many publications; see, e.g.,  \cite{tammer2015} and the references therein. The Gerstewitz approach is more general than the weight-sum one and allows us, in particular, to handle the multiobjective problem (\text{MO-$\ell_0$NP}) with {\em Lipschitz continuous} (not just smooth) functions $f_i$, which is important for the design and justification of subgradient-type methods vs.\ the gradient descent. To begin with, we review several properties of the Gerstewitz function needed for developing our algorithms. Some of these properties are known, but nevertheless their simplified proofs are provided for completeness and the reader's convenience. Our main novelty here is the choice and verification of the {\em directions} in the Gerstewitz function in order to adjust her scalarization technique to the nonsmooth and nonsmooth $\ell_0$-norm setting in multiobjective $\ell_0$ optimization, which will implemented in the subsequent sections.

\begin{definition}\label{def:tammer} Given a closed convex set $C \subset\R^m$, a nonzero direction $k^0 \in C \backslash (-C)$, and an nonempty set $A\subset\R^m$ with $A-\R_+ k^0 \subset A$, the {\sc Gerstewitz function} $\phi_A := \phi_{A,k^0} :\R^m\rightarrow \overline{\R}$, associated with $A$ and $k^0$ is defined by
\begin{align}\label{gerst}
\phi_{A,k^0}:= \inf\big\{t \in \R \mid y \in tk^0 +A\big\}.
\end{align}
\end{definition} 

Note that the choice of the set $A$ in Definition~\ref{def:tammer} provides some flexibility in scalarization and contains the weight-sum scalarization \eqref{weight} with different weight coefficients as special cases. In particular, the case where $A=\{(y_1,y_2)\in\R^2\;|\;y_1+y_2\le 0\}$ corresponds to \eqref{weight} with $(\omega_1,\omega_2)=(1/2,1/2)$, the choice $\{(y_1,y_2)\in\R^2\;|\;2y_1+y_2\le 0\}$ gives us \eqref{weight} with $(\omega_1,\omega_2)=(2/3,1/3)$. Other examples for the choice of the $A$, which provide nonsmooth scalarizations are given in Section~\ref{sec:numer}. Figure~\ref{G-map} illustrates the construction of the Gerstewitz function, where the set $A$ has a nonsmooth boundary.
\begin{figure}[H]
\centering
\includegraphics[width=0.7\textwidth]{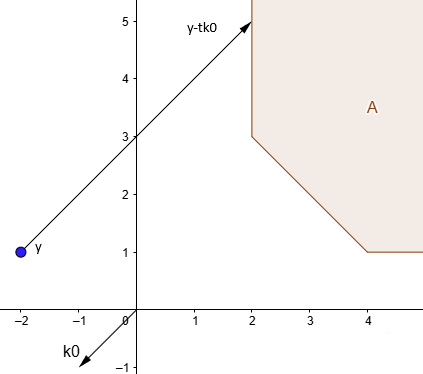}
\caption{Gerstewitz function with $A=\{(y_1,y_2) \mid y_1 \geq 2, y_2 \geq 1, y_1+y_2 \geq 5\}$}
\label{G-map}
\end{figure}

The first statement of this section provides conditions ensuring the properness and finiteness of the Gerstewitz scalarization function \eqref{gerst}.

\begin{proposition}\label{proper}
In addition the conditions on \eqref{gerst} in 
Definition~{\rm\ref{def:tammer}}, assume that $C$ is a cone. The following assertions hold:

{\bf(i)} $\phi_A$ is proper if and only if the set $A$ does not contain any line parallel to $k^0$, i.e.,
\begin{align}\label{proper1}
\mbox{ for all }\;y \in
\R^m\;\mbox{ there exists }\;t \in \R \;\mbox{ with }\; y+tk^0 \notin A.
\end{align}
        
{\bf(ii)} $\phi_A$ is finite if and only if $A$ does not contain any line parallel to $k^0$ and 
\begin{align*}
\R k^0 +A =\R^m.
\end{align*}
\end{proposition}
\begin{proof}
To verify (i), suppose that there exists $y$ such that 
\begin{align*}
\phi_A (y) = -\infty
\Longleftrightarrow y \in tk^0 +A\;\mbox{ for all }\;t\in \R
\Longleftrightarrow \{y+tk^0 \mid t \in \R\} \subset A.
\end{align*}
This directly implies that $\phi_A$ is proper if and only if \eqref{proper1} is satisfied.

To prove (ii), check first that if $A$ does not contain any line parallel to $k^0$, then $\phi_A$ is proper. Indeed, we get that $\dom \phi_A= \R k^0 +A=\R^m$, which ensures that $\phi_A(y) <\infty$, and thus $\phi_A$ is finite. Conversely, 
if $\phi_A$ is finite, then $\R^m\subset \dom\,\phi_A$ and the properness of $\phi_A$ implies by (i) that $A$ does not contain any line parallel to $k^0$. We clearly have $\R k^0 +A =\R^m$, which completes the proof of the proposition.
\end{proof}

The next theorem provides sufficient conditions ensuring {\em lower semicontinuity} (l.s.c.) and {\em continuity} of Gerstewitz function \eqref{gerst} and contains also a property crucial for applying \eqref{gerst} to $\ell_0$ multiobjective  optimization; see Remark~\ref{rem:gerst}.

\begin{theorem}\label{Theorem 2.1}
In addition to the assumptions  on \eqref{gerst} in 
Definition~{\rm\ref{def:tammer}}, suppose that $C$ is a cone and that $A$ is closed with $A\ne\R^m$. The following assertions hold:

{\bf(i)} If $A-C =A$, then $\phi_A$ is l.s.c.\ and we have
\begin{align}\label{Gerstewitz 2.1}
\big\{ y \in\R^m \mid \phi_A (y) \leq \lambda\big\} = \lambda k^0 +A\;\mbox{ for all }\;\lambda \in \R.
\end{align}

{\bf(ii)} If $A - (C\backslash \{0\})=\intset A$, then $\phi_A$ continuous and  we have
\begin{align}
&\big\{ y \in\R^m \mid \phi_A (y) < \lambda\big\} = \lambda k^0 +\intset A\;\mbox{ for all }\;\lambda \in \R, \label{Gerstewitz 2.2}\\
&\big\{ y \in\R^n \mid \phi_A (y) = \lambda\big\} = \lambda k^0 +\bd A\;\mbox{ for all }\;\lambda \in \R.\label{Gerstewitz 2.3}
\end{align}
\end{theorem}
\begin{proof}
To verify (i), denote $A': =\{(y,t) \in \R^m\times \R \mid y \in tk^0 +A\}$ and pick $(y,t) \in A'$ with $t' \geq t$. Our goal is to show that $(y,t') \in A'$. Indeed, we get that
\begin{align*}
t k^0 +A = t' k^0 +A -(t' -t) k^0  \subset t' k^0 +A,
\end{align*}
and hence $(y,t') \in A'$. Define $T:\R^m\times \R \rightarrow Y$ by $T(y,t): =y- t k^0$, which is a linear continuous mapping. It is clear that $A' =T^{-1} (A)$ and that the assumed closedness of $A$ yields the closedness of $A'$. Since $A' \subset \epi \phi_A \subset \cl A'$ and $A'$ is closed, we get $A' = \epi \phi_A$, which justifies that $\phi_A$ is l.s.c. Considering further $y \in \lambda k^0 +A$ and $\lambda \in \{t \in \R \mid y \in t k^0 +A\}$ tells us that $\lambda \geq \phi_A(y)$, i.e., $y \in \{z \in Y \mid \phi_A(z) \leq \lambda\}$, which justifies the inclusion ``$\supset$" in \eqref{Gerstewitz 2.1}. To verify the opposite inequality therein, take $t:=\phi_A(y) \leq \lambda$ and get $y \in t k^0 +A =\lambda k^0 +A -(\lambda -t) k^0 \subset A + \lambda k^0$. This gives us therefore that
\begin{align*}
\phi_A(y+\lambda k^0) &= \inf\big\{t \in \R \mid y + \lambda k^0 \in tk^0 +A\big\}\\
&= \inf\big\{t \in \R \mid y  \in (t-\lambda)k^0 +A\big\}.
\end{align*}
Denoting further $t':=t-\lambda$, we arrive at
\begin{align*}
\phi_A(y+\lambda k^0) &= \inf\big\{t' +\lambda \in \R \mid y  \in t'k^0 +A\big\}\\
&=\inf\big\{t' \in \R \mid y  \in t' k^0 +A\big\} +\lambda\\
&=\phi_A(y) + \lambda,
\end{align*}
which thus completes the verification of assertion (i).

To prove (ii), pick any $\lambda \in \R$ and choose $y \in \lambda k^0 + \intset A$. Since $y-t k^0 \in \intset A$, there exists $\epsilon >0$ with $y-tk^0 +\epsilon k^0 \in A$. which yields $\phi_A(y)\leq \lambda -\epsilon <\lambda$. This shows that the inclusion "$\supset$" holds in \ref{Gerstewitz 2.2}. To verify the reverse inclusion, let $\lambda \in \R$ and $y \in Y$ be such that $\phi_A(y) < \lambda$. There exists $t \in \R$ with $t< \lambda$ such that $y \in t k^0 +A$. Then $y \in \lambda k^0 +A -(\lambda-t)k^0 \subset \lambda k^0 + \intset A$, and hence \eqref{Gerstewitz 2.2} holds. Moreover, this shows that $\phi_A$ is upper semicontinuous. Combining the latter with (i) tells us that $\phi_A$ is continuous. Since $\bd A =A \backslash (\intset A)$, it follows from \eqref{Gerstewitz 2.1} and \eqref{Gerstewitz 2.2} that \eqref{Gerstewitz 2.3} is satisfied, which therefore completes the proof of the theorem.
\end{proof} 

\begin{remark}\label{rem:gerst}
{\rm It follows from the proof of Theorem~\ref{Theorem 2.1} that 
\begin{align}\label{key}
\phi_A(y+\lambda k^0) = \phi_A(y) +\lambda.
\end{align}
This is a {\em key property} of the Gerstewitz scalarization  to handle $\ell_0$ multiobjective optimization and develop a subgradient algorithm to find local Pareto solutions; see below.}
\end{remark}

Now we present useful characterizations of the convexity and positive homogeneity properties of the Gerstewitz scalarization function.

\begin{proposition}\label{G-conv} Suppose that all the assumptions of Theorem~{\rm\ref{Theorem 2.1}(i)} are satisfied. Then:

{\bf(i)} $\phi_A$ is a convex function if and only if $A$ is a convex set.
 
{\bf(ii)} $\phi_A$ is a positively homogeneous function if and only if $A$ is a cone. 
\end{proposition}   
\begin{proof}
We clearly have the representation $\epi\,\phi_A = T^{-1} (A)$, where the linear operator $T$ is defined in the proof of Theorem~\ref{Theorem 2.1}(i). Then the function $\phi_A$ is convex (resp. positively homogeneous) if and only if $A=T(\epi\,\phi_A)$ is a convex (resp.\ conic) set. 
\end{proof}

Given a closed, convex, and pointed cone $B\subset\R^m$, recall that the {\em partial order} relation $\leq_B$ on $\R^m$ induced by $B$ is defined by 
\begin{align}\label{order}
x \leq_B y \quad \text{if and only if} \quad y-x \in B. 
\end{align}
It is easy to check that that the partial order $\leq_B$ in \eqref{order} satisfies the properties:

$\bullet$ {\em Reflexivity}: $x \leq_B x$. 

$\bullet$ {\em Antisymmetry}: If $x \leq_B y$ and $y\leq_B x$, then $x=y$.

$\bullet$ {\em Transitivity}: If $x \leq_B y$ and $y \leq_B z$, then $x \leq_B z$. 

$\bullet$ {\em Convergence order}: If there are sequences $\{x_k\}$ and $\{y_k\}$ with $x_k\leq_B y_k$ for all $k \in \N$ such that $x_k \rightarrow x$ and $y_k\rightarrow y$ as $k \rightarrow \infty$, then $x \leq_B y$.\vspace*{0.05in}

The next proposition presents characterizations of the standard {\em monotonicity} property of the Gerstewitz function with respect to the partial order \eqref{order}; see the proof.

\begin{proposition}\label{monotonicity of Gerstewitz}
Let $A$ be a closed subset of $\R^m$ with $A\ne\R^m$, let $B\subset\R^m$ be a cone that induces the partial order \eqref{order}, and let $A-C=A$. Then the Gerstewitz function $\phi_A$ is monotone with respect to $B$ if and only if $A-B \subset A$.
\end{proposition}
\begin{proof}
First we check that the inclusion $A-B \subset A$ yields the monotonicity of $\phi_A$. Pick any $y_1, y_2 \in\R^m$ with $y_2 -y_1 \in B$ and choose $t\in \R$ such that $y_2 \in t k^0 +A$. Then we get $y_1 \in y_2 -B \subset t k^0 +(A-B) \subset t k^0+A$ and $\phi_A(y_1) \leq t$. Therefore, $\phi_A(y_1) \leq \phi_A(y_2)$, which justifies the monotonicity of the Gerstewitz function $\phi_A$ with respect to $B$. 

To verify the reverse implication, suppose that $\phi_A$ is monotone with respect to $B$ and then pick $y \in A$ and $b \in B$. Since $\phi_A (y) \leq 0$, $y-(y-b) \in B$, and $\phi_A$ is monotone with respect to $B$, we have the inequalities
\begin{align*}
\phi_A (y-b) \leq \phi_A (y) \leq 0.
\end{align*}
The latter shows that $A-B \subset A$ and thus completes the proof.
\end{proof}

Now we derive a simple condition, which ensures the Lipschitz continuity of the Gerstewitz function on the entire space $\R^m$. The property plays an important role in verifying the convergence of our subgradient algorithm proposed below.\vspace*{0.03in}

In what follows, we {\em identify the cone} $C$ in definition \eqref{gerst} of the Gerstewitz function with the {\em ordering cone} $B$ in \eqref{order}. In these terms, recall the useful relationship from \cite{tammer2014lipschitz}: 
\begin{align}\label{lip1}
\phi_A(y) \leq \phi_A(y') +\phi_{-C}(y-y')\;\mbox{ for all }\;y, y' \in\R^m.
\end{align}

\begin{theorem}\label{Lipschitz Property}
If $k^0 \in \intset C$, then $\phi_A=\phi_{A,k^0}$ is Lipschitz continuous on $\R^m$.
\end{theorem}
\begin{proof}
Given $k^0 \in \intset C$, let $V \subset\R^m$ be a closed ball centered at $0$ such that $k^0+V \subset C$, and let $p_V: Y \rightarrow \R$ be the Minkowski functional associated with $V$. It is well known that $p_V$ is a continuous seminorm and that $V=\{y\in\R^m\mid p_V(y) \leq 1\}$. Pick $y \in\R^m$ and $t>0$ such that $y \in t V$. Then $t^{-1} y \in V \subset k^0 -C$, and so $y \in tk^0-C$ ensuring that $\phi_{-C}(y) \leq t$. Therefore, $\phi_{-C} (y) \leq p_V(y)$, which confirms that $\R k^0= \dom \phi_{-C} =\R^m$. It follows from the convexity and positive homogeneity  of $\phi_{C}$ that
\begin{align*}
\phi_{-C} (y) \leq \phi_{-C} (y') +p_V(y-y'),
\end{align*}
which implies the inequality
\begin{align}\label{inequality (6), Lipschitz prop}
\vert \phi_{-C} (y)-\phi_{-C} (y') \vert \leq p_V (y-y')\;\mbox{ for all }\;y, y' \in\R^m.
\end{align}
This justifies the Lipschitz continuity of $\phi_{-C}$. Let us show  that
$\phi_A$ does not take the value $-\infty$. If on the contrary $\phi_A (y_0) = -\infty$ for some $y_0 \in\R^m$, then $y_0 + \R k^0 \subset A$ yielding
\begin{align*}
A=A-C \supset y_0 +\R k^0 -C=y_0 +\R^m =\R^m,
\end{align*}
a contradiction. Since $\dom\,\phi_A =\R k^0 +A = \R k^0 -C+A =\R^m+A=\R^m$ and the assumptions of Proposition~\ref{proper} are clearly satisfied, we deduce from assertion (ii) therein that $\phi_A$ is finite. Using finally \eqref{lip1} together with \eqref{inequality (6), Lipschitz prop} tells us that 
\begin{align*}
\vert \phi_A (y) - \phi_A (y') \vert \leq p_V (y-y')\;\mbox{ for all }\;y, y' \in\R^m, 
\end{align*}
which justifies the Lipschitz continuity of $\phi_A$ on $\R^m$ and thus completes the proof.
\end{proof}

We need the (convex) subdifferential calculation taken from 
\cite[Corollary~4.2]{tammer2014lipschitz}, where $A_\infty$ stands for the horizon/recession  cone associated with the convex set $A$; see \cite{rockafellar1998}.

\begin{proposition}\label{gert-sub} Let $A$ be a convex set, and let $k^0 \notin A_\infty$. Then for all $\bar{y}\in\R^m$ we have 
\begin{align*}
\partial\phi_A(\bar{y}) =\big\{y^* \in {\rm cl}\,A\mid \langle k^0,y^* \rangle=1,\; \langle\bar{y},\;y^* \rangle -\phi_A (\bar{y}) \geq \langle y, y^* \rangle\;\mbox{ whenever }\;y \in A\big\},
\end{align*}
where ${\rm cl}\,A$ signifies the closure of the set $A$.
\end{proposition} 

\section{$\ell_0$ Optimization via Gerstewitz Scalarization}\label{sec:G-l0}

Now we are in a position to define the {\em scalarized $\ell_0$ Gerstewitz function} associated with problem (MO-$\ell_0$NP) as follows:
\begin{align}\label{G-l0}
\phi_A(f_1(x) + \Vert x\Vert_0, \ldots, f_m(x) + \Vert x\Vert_0),\quad x\in\R^n.
\end{align}
Choosing $k^0= (1, \ldots, 1) \in \R^m$, assuming that $\Vert x\Vert_0 =\lambda$, and using the key property \eqref{key} allow us to rewrite \eqref{G-l0} in the equivalent form 
\begin{align*}
\phi_A(f_1(x) + \Vert x\Vert_0, \ldots, f_m(x) + \Vert x\Vert_0)&= \phi_A((f_1(x) , \ldots, f_m(x)) +  \lambda k^0) \\
&= \phi_A(f_1(x) , \ldots, f_m(x)) + \lambda \\
&= \phi_A(f_1(x) , \ldots, f_m(x)) + \Vert x \Vert_0.
\end{align*}
This leads us to the {\em Gerstewitz-scalarized} (MO-$\ell_0$NP) {\em problem} defined by
\begin{align}\label{G-multi}
\min  \phi_A(f_1(x) , \ldots, f_m(x)) + \Vert x \Vert_0,\quad x\in\R^n.
\end{align}

Observe that the above choice of the direction $k^0$ is instrumental to take the $\ell_0$-norm out of the composition in the Gerstewitz function $\phi_A$ as in \eqref{G-l0} and place the $\ell_0$-norm as an {\em additive term} in \eqref{G-multi}. This plays a crucial role in the development and justification of our algorithm to solve the multiobjective problem (MO-$\ell_0$NP) in what follows.

To deal with the Gerstewitz composition $\phi_A(f_1(x) ,\ldots, f_m(x))$ in \eqref{G-multi}. We need the following two major properties. The first one concerns the convexity of $\phi_A(f_1(x) , \ldots, f_m(x))$ with respect to the {\em nonpositive cone partial order} on $\R^m$; cf.\ Proposition~\ref{monotonicity of Gerstewitz}.

\begin{proposition}\label{multi-mon} In the setting of \eqref{G-multi}, assume that the functions $f_i :\R^n \rightarrow \R$ for $i=1,\ldots,m$ are convex and the set $A\ne\R^m$ is convex being such that $A-C =A$  and $A- \R^m_+ \subset A$. Then $\phi_A (f_1, \ldots, f_m)$ is convex with respect to the partial order $\le_{\R^{m}_+}$ in \eqref{order}.
\end{proposition}
\begin{proof} By the convexity of $f_i$, we get 
\begin{align*}
f_i(\lambda x + (1-\lambda) y) \leq \lambda f_i(x) + (1-\lambda)f_i
\end{align*}
for all $x,y\in\R^m$ and $i=1,\ldots,m,\;\lambda \in [0,1]$. Therefore,
\begin{align*}
 &(f_1(\lambda x + (1-\lambda) y), \ldots, f_m(\lambda x + (1-\lambda) y)) \\
&\le_{\R^{m}_+}\lambda (f_1(x), \ldots,f_m(x)) + (1-\lambda) (f_1(y), \ldots, f_m(y)).
\end{align*}
It follows from Proposition~\ref{monotonicity of Gerstewitz} that
\begin{align*}
&(\phi_A ((f_1(\lambda x + (1-\lambda) y), \ldots, f_m(\lambda x + (1-\lambda) y))\\
&\le \phi_A (\lambda (f_1(x), \ldots,f_m(x)) + (1-\lambda) (f_1(y), \ldots, f_m(y))),
\end{align*}
which completes the proof of the proposition.
\end{proof}

The second proposition concerns the Lipschitz continuity of the Gerstewitz composition  $\phi_A(f_1(x) , \ldots, f_m(x))$ in the $\ell_0$ optimization problem \eqref{G-multi}. 

\begin{proposition}
Assume that the function $f_i$, $i=1,\ldots,m$, are Lipschitz continuous on $\R^n$. Then the Gerstewitz composition
$\phi_A (f_1, \ldots, f_m)$ is also Lipschitz continuous on $\R^n$. 
\end{proposition} 
\begin{proof}
Let $L_i$ be the Lipschitz constants of $f_i$ for $i=1,\ldots,m$. Then we have 
\begin{align*}
\Vert (f_1 (x), \ldots, f_m(x)) -(f_1 (y), \ldots, f_m(y)) \Vert \leq \sum_{i=1}^m L_i \Vert x -y\Vert.
\end{align*}
It follows from Theorem~\ref{Lipschitz Property} that, under the choice of the direction $k^0$ above, the Gerstewitz function $\phi_{A,k^0}$ is Lipschitz continuous on $\R^m$ with some constant $M$. Having all of this together leads us to the inequalities
\begin{align*}
&\Vert\phi_A((f_1 (x), \ldots, f_m(x))) - \phi_A((f_1 (y), \ldots, f_m(y)) \Vert \\
&\le M\Vert (f_1 (x), \ldots, f_m(x)) - (f_1 (y), \ldots, f_m(y)) \Vert \\
&\le M\sum_{i=1}^m L_i \Vert x -y\Vert,
\end{align*}
which justifies the claimed Lipschitz continuity of the composition in \eqref{G-multi}.
\end{proof}

\section{Algorithms for $\ell_0$ Optimization}\label{sec:algor}

In this section, we design our novel algorithms to solve single-objection and multiobjective optimization problems containing the $\ell_0$-norm function. Our algorithms are of the subgradient type that are based, due to the heavy nonconvexity of the $\ell_0$-norm function, on the complete calculation of its limiting subdifferential given in Theorem~\ref{sub-cal}. 

Our further strategy is as follows. First we design two algorithms to solve the scalar problem of $\ell_0$ optimization (\rm{$\ell_0$NOP}). Considering $\ell_0$ multiobjective optimization, the results above allow us to find their Pareto solutions by employing the two scalarization approaches via  the weight-sum and Gerstewitz functions. Accordingly, we get the two subgradient-type algorithms for the $\ell_0$ multiobjective optimization problem while dealing with its scalarized versions in \eqref{weight} and \eqref{G-multi}, respectively,

The first algorithm to find local minimizers of the scalar $\ell_0$ optimization problem ($\ell_0$NOP) by using the subdifferential calculation in Theorem~\ref{sub-cal} with the projection notation described in Section~\ref{sec:optim}. Observe that the term $\tilde{I}_k \nabla f(x_k)$ in the algorithm is the gradient of the function $f$ over the projection $(f \circ \tilde{I}_k)$ at $x_k$, and we have 
\begin{align*}
\tilde{I}_k \nabla f(x_k) = \nabla f(x_k) + (v_{1,k}, \ldots, v_{n,k}) \in \nabla f(x_k) +\partial \Vert \cdot \Vert_0(x_k), 
\end{align*}
where $v_{i,k}$ are defined by
$$ 
v_{i,k}:=
\begin{cases}
\dfrac{\partial f}{\partial  x_i}(x_{i,k}) \quad &\text{if} \quad x_{i,k} =0,\\
0 \quad &\text{if} \quad x_{i,k} \neq 0.
\end{cases}$$
Clearly, this choice of $v$ is a subgradient of the $\ell_0$-norm function.

\begin{algorithm}[H]
\caption{to find a local minimizer in $\ell_0$ scalar optimization}\label{alg:cap}
\begin{algorithmic}
\Require Given $x_0$ as starting point, $\epsilon >0$, stepsize $t < \dfrac{1}{L}$
\If {$\Vert \tilde{I}_k \nabla f(x_k) \Vert \geq \epsilon$}	
\\ 	$x_{k+1} =x_k -t \cdot \tilde{I}_k \nabla f(x_k)$
\Else \quad Stop 
\EndIf 	
\end{algorithmic}
\end{algorithm}

Algorithm~\ref{alg:cap} ensures that if $x_k$ is in one hyperplane at some $k$-th iteration, then all $x_j$ are in the same hyperplane for every $j >k$. Observe also that we get the inclusion  $\tilde{I}_j \subset\tilde{I}_k$ by taking into account that for $x_{i,k}=0,$ the vector $\tilde{I}_k \nabla f(x_k)$ has the same zero value $i$-th component. By Theorem~\ref{theorem of local min}, it is important to consider each hyperplane $I_k$, since some of our local minimizers lie on these hyperplanes.

We see that Algorithm~\ref{alg:cap} stops at a local minimizer. However, problem ($\ell_0$NOP) may have many local minimizers, and it may be of interest to find other local minimizers of the problem as well. The second algorithm allows us to proceed further in this direction where $I$ stands for the identity matrix.

\begin{algorithm}[H]
\caption{to find  local multiply minimizers of ($\ell_0$NOP)}\label{alg:cap0}
\begin{algorithmic}
\Require Given $x_0$ as starting point, $\epsilon >0$, stepsize $t <\dfrac{1}{L}$
\If {$\Vert \nabla f(x_k) \Vert \geq \epsilon$ }
\If {$\Vert \tilde{I}_k \nabla f(x_k) \Vert \geq \epsilon$}	
\\\quad $x_{k+1} =x_k -t \cdot \tilde{I}_k \nabla f(x_k)$
\Else\\ \quad $x_{k+1} =x_k -t \cdot (I-\tilde{I}_k) \nabla f(x_k)$
\EndIf 
\Else \quad Stop
\EndIf	
\end{algorithmic}
\end{algorithm}

Algorithm~\ref{alg:cap0} can be interpreted as follows. Given  any starting point, we employ Algorithm~\ref{alg:cap}, which stops at some local minimizer. Then Algorithm~\ref{alg:cap0} forces the iteration out of the projection that the obtained local minimizer is in by using the matrix  $I - \tilde I_k$. This  creates a new starting point and gives us a new direction to reach a new local minimizer. In this way, we may eventually arrive at a global solution to the problem.\vspace*{0.03in}

The next algorithm addresses solving the multiobjective $\ell_0$ optimization problem via the weight-sum scalarization.

\begin{algorithm}[H]
\caption{$\ell_0$ multiobjective optimization via weight-sum scalarization}\label{alg:cap1}
\begin{algorithmic}
\Require Given $x_0$ as starting point, $\epsilon >0$, stepsize $t < \dfrac{1}{L}$, where $L:= \max_{i=1,\ldots,m} {L_i}$, choose $\omega_i$ in {\rm weight sum} \eqref{weight}
\If {$\Vert \tilde{I}_k \nabla({\rm weight\;sum})(x_k) \Vert \geq \epsilon$}	
\\ 	$x_{k+1} =x_k -t \cdot \tilde{I}_k \nabla({\rm
weight\;sum})(x_k)$
\Else \quad Stop 
\EndIf 	
\end{algorithmic}
\end{algorithm}

The last algorithm addresses finding local minimizers of the Gerstewitz-scalarized $\ell_0$ minimization problem \eqref{G-multi} with arriving at a local Pareto solution to the $\ell_0$ multiobjective optimization problem (MO-{$\ell_0$}NP) under a special choice of the set $A$.

\begin{algorithm}[H]
\caption{$\ell_0$ multiobjective optimization via Gerstewitz scalarization}\label{alg:cap2}
\begin{algorithmic}
\Require Given $x_0$ as starting point, $\epsilon >0$, step size $t$, $A \subset\R^m$ as a convex set with $A-\R^m_+ \subset A$, $k^0 = (1, \ldots, 1)$
\If {$\Vert \phi_A (f_1, \ldots ,f_m)(x^{(k)}+1) - \phi_A (f_1, \ldots ,f_m)(x^{(k)})\Vert \geq \epsilon$}	
\\ 	$x_{k+1} =x_k -t \cdot \tilde{I}_k J_f(x^{(k)})^\top g^{(k)}$, where $J_f(x^{(k)})$ is a Jacobian matrix at $x^{(k)}$ and \\ $g^{(k)} \in \partial \phi_A (f_1 (x^{(k)}), \ldots f_m(x^{(k)}))$\\
$\phi_A \circ f_{best} = \min \{\phi_A \circ f_{best}, \phi_A \circ f(x^{k+1})\}$
\Else \quad Stop 
\EndIf 	
\end{algorithmic}
\end{algorithm}

\section{Convergence Analysis}\label{sec:conver}

This section conducts convergence analysis of the proposed algorithms to find local solutions to the $\ell_0$ of single-objective and multiobjective optimization problems. It is natural to start with scalar $\ell_0$ optimization while restricting our attention to the convergence proof for Algorithm~\ref{alg:cap}. The convergence proof for Algorithm~\ref{alg:cap0} can be done similarly, and we leave the details to the reader. \vspace*{0.03in}

To begin with, let us first verify the {\em monotonicity property} of Algorithm~\ref{alg:cap}, which is crucial for the proof of convergence while being of its own interest.

\begin{theorem}\label{Monotone}
Under the standing assumptions on the function $f$ formulated in 
Section~{\rm\ref{sec:optim}}, We have the following monotonicity property of iterates in Algorithm~{\rm\ref{alg:cap}} with respect to the objective function of problem {\rm({$\ell_0$NOP})}:
\begin{align*}
f(x_{k+1}) +\Vert x_{k+1} \Vert_0 \leq f(x_k) +\Vert x_k \Vert_0.
\end{align*}
\end{theorem}
\begin{proof}
It follows from the convexity and smoothness of $f$ with the Lipschitz gradient that
\begin{align*}
f(y) &\leq f(x) + \nabla f(x)^\top (y-x) + \dfrac{1}{2} \nabla^2 f(x) \Vert y-x \Vert_2^2\\
&\leq f(x) + \nabla f(x)^\top (y-x) + \dfrac{1}{2} L \Vert y-x \Vert_2^2,
\end{align*}
where $L$ is a Lipschitz constant of $f$. Putting there $y=x_{k+1}= x_k -t \cdot \tilde{I}_k \nabla f(x_k)$ and $x=x_k$ brings us to the estimate
\begin{align}
f(x_{k+1}) &\leq f(x_k) + \nabla f(x_k)^\top  (x_{k+1}-x_k) + \dfrac{1}{2} L \Vert x_{k+1}-x_k \Vert_2^2 \nonumber\\
&=f(x_k) + \nabla f(x_k) (x_k -t \cdot \tilde{I}_k \nabla f(x_k)-x_k) + \dfrac{1}{2} L \Vert x_k -t \cdot \tilde{I}_k \nabla f(x_k)-x_k \Vert_2^2 \nonumber\\
 &=f(x_k) - \nabla f(x_k) t \cdot \tilde{I}_k \nabla f(x_k) + \dfrac{1}{2} L \Vert t \cdot \tilde{I}_k \nabla f(x_k) \Vert_2^2 \nonumber\\
 &=f(x_k) - t   \Vert \tilde{I}_k \cdot\nabla f(x_k) \Vert_2^2 + \dfrac{1}{2} Lt^2 \Vert  \cdot \tilde{I}_k \nabla f(x_k) \Vert_2^2 \nonumber\\
&=f(x_k)  - t\left(1- \dfrac{1}{2} Lt \right) \Vert  \tilde{I}_k \nabla f(x_k) \Vert_2^2.\nonumber
\end{align}
Using $t\leq \dfrac{1}{L}$, we know that $- \left(1- \dfrac{1}{2} Lt \right)=\dfrac{1}{2} Lt -1 \leq \dfrac{1}{2} L \cdot \dfrac{1}{L} -1 =\dfrac{1}{2}-1=-\dfrac{1}{2}$. Plugging the latter into the last inequality tells us that  
\begin{align}\label{eq1}
f(x_{k+1}) \leq f(x_k) -\dfrac{1}{2} t \Vert  \tilde{I}_k \nabla f(x_k) \Vert_2^2.
\end{align}
Since $\dfrac{1}{2} t \Vert  \tilde{I}_k \nabla f(x_k) \Vert_2^2 > 0$ unless $\tilde{I}_k\nabla f(x_k)=0$, which is the stopping condition for optimal solutions to problem $(PP)$ in Theorem~\ref{theorem of local min}, we get
\begin{align}\label{f-mon}
f(x_{k+1}) \leq f(x_k).
\end{align}
As discussed  after the formulation of Algorithm~\ref{alg:cap}, it follows from $\tilde{I}_j \subset\tilde{I}_k$ for $j>k$ that $\Vert x_j \Vert_0 \leq \Vert x_k \Vert_0$, and  hence $\Vert x_{k+1} \Vert_0 \leq \Vert x_{k} \Vert_0$. Combining with \eqref{f-mon} yields
\begin{align*}
f(x_{k+1}) + \Vert x_{k+1} \Vert_0  \leq f(x_k)+ \Vert x_{k} \Vert_0,
\end{align*}
which therefore completes the proof of the theorem.
\end{proof}

Now we are ready to establish the convergence of Algorithm~\ref{alg:cap} to a local minimizer of problem ($\ell_0$NOP) with a rate estimate. 

\begin{theorem}\label{Convergence} Let $x_k$ be the $k$-th iteration of Algorithm~{\rm \ref{alg:cap}} of Theorem~{\rm\ref{Monotone}}, and let $I_k$ be the associated smallest subspace. Then the sequence $\{x_k\}$ converges to a local minimizer $\bar x$ of problem {\rm($\ell_0$NOP)} with the rate estimate
\begin{align*}
f(x_{(k+s)}) +\Vert x_{(k+s)} \Vert -f(\bar x) -\Vert \bar x \Vert \leq \dfrac{\Vert x_{k} -\bar x\Vert_2 ^2}{2st}.
\end{align*}
\end{theorem}
\begin{proof}
The iterative process stops after $k$-th iteration if we have $x_k \in I_k$ for the corresponding hyperplane. Our $\ell_0$ optimization problem ($\ell_0$OP) turns into a projection problem $(PP)$, which is the convex problem by the convexity of the function $f \circ \tilde{I}_k$. Starting from this iteration, Algorithm~\ref{alg:cap} becomes the classical gradient descent method with $x_k$ being the starting point of the algorithm, and thus its the convergence follows.

To proceed in this direction, let $x_k$ be the starting point of subproblem $(PP)$. The assumption that $x_k$ is the $k$-th iteration of ($\ell_0$OP) with $I_k$ being the associated smallest subspace ensures that $\Vert x_{k+j} \Vert_0 =\Vert x_{k} \Vert_0$ for all $j$. Since $(PP)$ is a convex problem, we have
\begin{align*}
f(\bar x) -f(x) \geq \nabla f(x)^\top(x-\bar x),\\
f(x) \leq f(\bar x) + \nabla f(x)^\top (x-\bar x).
\end{align*}
Denote $x^+:= x-t\cdot\tilde{I}(x) \nabla f(x)$, where $\tilde{I}(x)$ signifies the projection matrix associated with $x$, and deduce from (\ref{eq1}) the relationships
\begin{align*}
f(x_{+}) &\leq f(x) -\dfrac{1}{2} t \Vert  \tilde{I}(x) \nabla f(x) \Vert_2^2\\
&\leq f(\bar x) + \nabla f(x)^\top (x-\bar x) -\dfrac{1}{2} t \Vert  \tilde{I}(x) \nabla f(x) \Vert_2^2\\
&=f(\bar x) + \dfrac{1}{2t}  \left(2t \nabla f(x)^\top (x-\bar x) -t^2 \Vert  \tilde{I}(x) \nabla f(x) \Vert_2^2 - \Vert x-\bar x\Vert^2_2+ \Vert x-\bar x\Vert^2_2\right)\\
&=f(\bar x) + \dfrac{1}{2t}  \left(\Vert x-\bar x\Vert^2_2   - \Vert x- t\tilde{I}(x) \nabla f(x)-\bar x\Vert^2_2 \right)\\
&= f(\bar x)+ \dfrac{1}{2t}  \left(\Vert x-\bar x\Vert^2_2 -\Vert x^+-\bar x\Vert^2_2\right).
\end{align*}
Summing over iterations leads us to the estimates 
\begin{align*}
\sum_{i=1}^s (f(x_{k+i}) -f(\bar x)) &\leq  \sum_{i=1}^s \dfrac{1}{2t} \left(\Vert x_{(k+i-1)}-\bar x\Vert^2_2 -\Vert x_{k+i}-\bar x\Vert^2_2\right)\\
&= \dfrac{1}{2t} \left( \Vert x_{k}-\bar x\Vert^2_2 -\Vert x_{k+s}-\bar x\Vert^2_2\right)\\
&\leq \dfrac{1}{2t} \Vert x_{k}-\bar x\Vert^2_2.
\end{align*}
It follows from Theorem~\ref{Monotone} that 
\begin{align*}
f(x_{k+s}) -f(\bar x) &\leq \dfrac{1}{s} \sum_{i=1}^s (f(x_{k+i}) -f(\bar x))\\
&\leq \dfrac{\Vert x_{k} -\bar x\Vert_2 ^2}{2st}
\end{align*}
which therefore completes the proof of the theorem.
\end{proof}

The next theorem verifies the convergence and the rate estimates for Algorithm~\ref{alg:cap0} to find a local Pareto solution of the $\ell_0$ multiobjective optimization problem (MO-{$\ell_0$}NP).

\begin{theorem}\label{alg:weight}
Consider problem {\rm(MO-{$\ell_0$}NP)} under the standing assumptions.
Then Algorithm~{\rm\ref{alg:cap1}} converges to a local Pareto solution of {\rm(MO-{$\ell_0$}NP)} with the rate estimate
\begin{align*}
{\rm(weight\;sum)}(x_{(k+s)}) +\Vert x_{(k+s)} \Vert {\rm-(weight\;sum)}(\bar x) -\Vert \bar x \Vert \leq \dfrac{\Vert x_{k} -\bar x\Vert_2 ^2}{2s t},
\end{align*}
\end{theorem}
\begin{proof}
Proposition~\ref{Pareto solution} tells us that we can find a local Pareto solution to problem (MO-{$\ell_0$}NP) as a local minimizer of the weight-sum $\ell_0$ minimization problem \eqref{weight}. Then the claimed results follow from Theorem~\ref{Convergence} applied to \eqref{weight}.
\end{proof}

The last theorem of this section verifies the convergence and the rate estimates for Algorithm~\ref{alg:cap2} to find a local minimizer of the Gerstewitz-scalarized $l^0$ multiobjective optimization problem in \eqref{G-multi} under the general choice of the set $A$ therein. A particular choice of $A$ allows us to find a local Pareto solution to the $\ell_0$ multiobjective problem (MO-{$\ell_0$}NP). The proof of this theorem is based on the properties of the Gerstewitz scalarization function \eqref{gerst} established in Sections~\ref{sec:tammer} and \ref{sec:G-l0}.

\begin{theorem}
Consider the Gerstewitz-scalarized $\ell_0$ optimization problem \eqref{G-multi} in the setting of Algorithm~{\rm\ref{alg:cap2}}. Then this algorithm converges to a local minimizer $\bar x$ of \eqref{G-multi} with the following rate estimate:
\begin{align}\label{rate}
\phi_A \circ f_\text{best}^{(k+s)} - \phi_A \circ\bar f\leq \dfrac{{\rm dist}(x^{(1+s)}, \bar x)^2 + M^2 t^2 k}{2t k},
\end{align}
where $M$ is a Lipschitz modulus of $\phi_A \circ f$, and where $\bar f:=f(\bar x)$. Moreover, the choice of the set $A$ satisfying $A-\R^{m}_+\subset A$ ensures that $\bar x$ is a local Pareto solution to the $\ell_0$ multiobjective optimization problem {\rm(MO-{$\ell_0$}NP)}.
\end{theorem}
\begin{proof}
Similarly to proof of Theorem~\ref{Convergence}, suppose that the process of reducing the dimension stops at the $s-$th iteration. Let $\bar x$ be any optimal solution in the smallest hyperplane over the iterations. To simplify the proof, denote $\tilde{g}^{(k)}: = J_f(x^{(k)})^\top g^{(k)}$ and get
\begin{align}\label{sub-ine}
\begin{array}{ll}
\Vert x^{(k+s1)} -\bar x \Vert_2^2 &=\Vert x^{(k+s)} -t \cdot \tilde{I}_{k+s} \tilde{g}^{(k+s)}  -\bar x \Vert_2^2\\
&=\Vert x^{(k+s)} -\bar x \Vert_2^2 -2t\cdot \tilde{I}_{k+s} \tilde{g}^{(k+s)}(x^{(k+s)} -\bar x) + t^2 \Vert \tilde{g}^{(k+s)} \Vert ^2_2\\
&\leq \Vert x^{k+s} -\bar x \Vert_2^2 - 2t(\phi_A \circ f(x^{(k+s)}) -\phi_A \circ\bar f) +t^2 \Vert g^{(k+s)} \Vert _2^2 
\end{array}
\end{align}
with the subgradient inequality coming from the definition
\begin{align*}
t\cdot \tilde{I}_{k+s} \tilde{g}^{(k+s)}(\bar x -x^{(k+s)}) &= t\cdot \tilde{g}^{(k+s)}(\bar x -x^{(k+s)})\\
&\leq \phi_A \circ\bar f -\phi_A \circ f(x^{(k+s)}). 
\end{align*}
Applying \eqref{sub-ine} recursively leads us to the estimate
\begin{align*}
\Vert x^{(k+s+1)} -\bar x \Vert_2^2 \leq  \Vert x^{(s+1)} -\bar x \Vert_2^2 -2 t \sum_{i=1}^k (\phi_A \circ f (x^{(i+s)} -\phi_A \circ\bar f) + t^2 \sum_{i=1}^k \Vert \tilde{g}^{(i+s)} \Vert_2^2.
\end{align*}
The usage of $\Vert x^{(k+s+1)} -\bar x \Vert_2^2 \geq 0$ yields 
\begin{align*}
2 t \sum_{i=1}^k (\phi_A \circ f (x^{(i+s)} -\phi_A \circ\bar f) \leq \Vert x^{(s+1)} -\bar x \Vert_2^2+ t^2 \sum_{i=1}^k \Vert \tilde{g}^{(i+s)} \Vert_2^2,
\end{align*}
which being combined with the inequality
\begin{align*}
t\sum_{i=1}^k (\phi_A \circ f (x^{(i+s)} -\phi_A \circ\bar f \geq  (\sum_{i=1}^k t) (\phi_A \circ f_{best}^{(k+s)} - f(\bar x))
\end{align*}
results in the composition estimate
\begin{align*}
\phi_A \circ f_{best}^{(k+s)} -\phi_A \circ\bar f \leq \dfrac{\Vert x^{(s+1)} -\bar x\Vert_2^2+ t^2 \sum_{i=1}^k \Vert \tilde{g}^{(k+s)}\Vert_2^2 }{2  \sum_{i=1}^k t}.
\end{align*}
Furthermore, the imposed Lipschitz continuity $\Vert \tilde{g}^{(k+s)}\Vert_2\leq M$ ensures that
\begin{align*}
\phi_A \circ f_{best}^{(k+s)} -\phi_A\bar f \leq \dfrac{\Vert x^{(s+1)} -\bar x\Vert_2^2+ t^2 k M^2 }{2 tk},
\end{align*}
and therefore we arrive at the condition
\begin{align*}
\phi_A \circ f_\text{best}^{(k+s)} - \phi_A \circ\bar f\leq \dfrac{\text{dist} (x^{(1+s)}, \bar x)^2 + M^2 t^2 k}{2tk},
\end{align*}
which verifies the convergence of Algorithm~\ref{alg:cap2} to the local minimizer $\bar x$ of \eqref{G-multi} with the claimed rate estimate in \eqref{rate}.
        
Finally, the choice of $A$ such that $A-\R^{m}_+\subset A$ and the monotonicity of the Gerstewitz function with respect to partial order ensure that $\bar x$ is a local Pareto solution to the $\ell_0$ multiobjective optimization problem (MO-{$\ell_0$}NP). This completes the proof.   
\end{proof}
 
\section{Numerical Illustrations}\label{sec:numer}

In this section, we demonstrate the performance of the designed algorithms by considering typical examples. All of the calculations were conducted by using Jupyter Notebook. 

\begin{example}
Let $f(x,y):= x^2 +2y^2 -2x -2xy+3 + \Vert (x,y) \Vert_0$ for $(x,y)\in\R^2$.
\end{example}
It is easy to check that the function in this example has 3 local minimizers: $(2,1), (1,0),\\ (0,0)$. We will use 3 different starting points to compare the  convergences of Algorithm~\ref{alg:cap} to optimal points and to see us how the algorithm performs in these settings, which is graphically illustrated in the figures below.
\begin{figure}[H]
\centering
\includegraphics[width=0.45\textwidth]{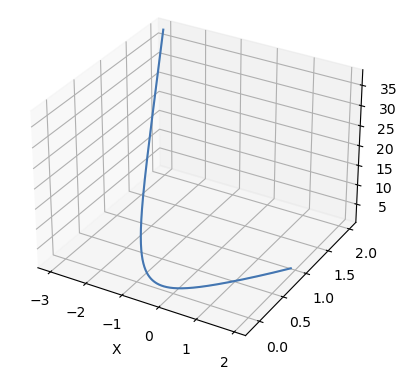}
\caption{Initial point $(-3,2)$}
\label{(-3,2)}
\end{figure}
\begin{figure}[H]
\centering
\includegraphics[width=0.45\textwidth]{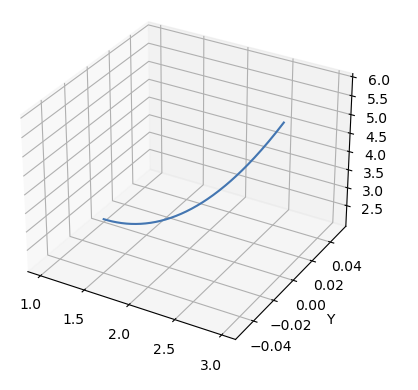}
\caption{Initial point $(3,0)$}
\label{(3,0)}
\end{figure}
\begin{figure}[H]
\centering
\includegraphics[width=0.45\textwidth]{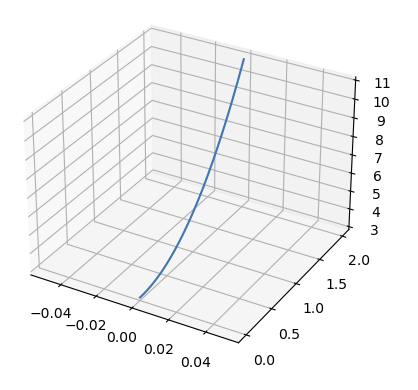}
\caption{Initial point $(0,2)$}
\label{(0,2)}
\end{figure}
As we can see in the the Figure~\ref{(3,0)} and Figure~\ref{(0,2)}, the iterations remain in the hyperplanes and converge to the corresponding local minimizers.

Let us further use an alternative definition of the $\ell_0$-norm function to demonstrate the possibility of the components going to zero. Instead of defining the $\ell_0$-norm by the condition that $x_i$ exactly equal to zero, we modify it by $\vert x_i \vert \leq \epsilon$, with $\epsilon = 10^{-6}$ in the next figure, to see the differences. 
\begin{figure}[H]
\centering
\includegraphics[width=0.45\textwidth]{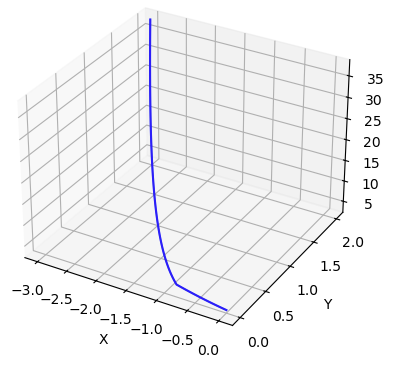}
\caption{Initial point $(-3,2)$ with alternative definition}
\label{(-3,2)(2)}
\end{figure}
Compared to Figure~\ref{(-3,2)}, which is similar to the classical gradient descent method, the modified iterations go into the hyperplane $y=0$ with the subsequent iterations remaining in it. To see the drop of the value and guarantee that Algorithm~\ref{alg:cap} is a descent algorithm, we can check the graph of the function values.
\begin{figure}[H]
\centering
\includegraphics[width=0.7\textwidth]{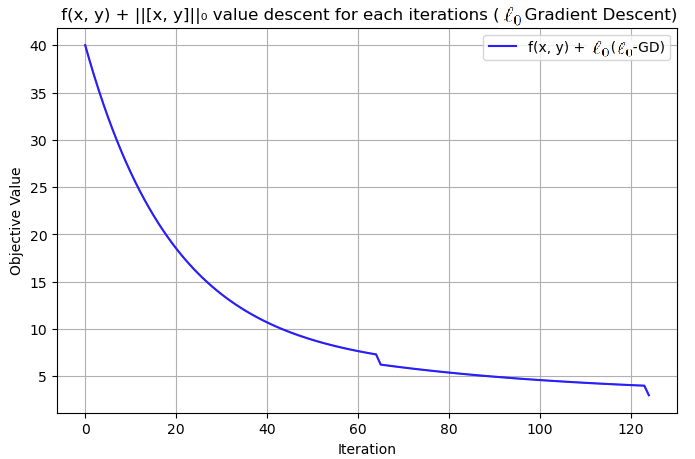}
\caption{Adaptive definition for the $\ell_0$ gradient descent}
\label{function value}
\end{figure}

The next example addresses $\ell_0$ multiobjective optimization problems. 

\begin{example}
Let $f_1(x,y):=(x-1)^2+y^2$ and $f_2(x) := x^2 +(y-2)^2$. 
\end{example}
 Apply the Gerstewitz scalarization, we choose the set $A=\{(y_1,y_2) \in \R^2 \mid y_1+y_2\leq 0\}$ with $k^0=(1,1)$ in Algorithm~\ref{alg:cap2}. Definition \eqref{gerst}
 of Gerstewitz function gives us $\phi_A (f_1,f_2) = \inf \{t \mid (f_1,f_2) -t k^0 \in A\}$, and so $f_1 -t +f_2-t \leq 0$, which yields $t \geq \dfrac{f_1 +f_2}{2}$. Taking the
 infimum of $t$, we get $t= \dfrac{f_1 +f_2}{2}$. Therefore, this scalarization agrees with to the case of the weighted-sum method.
\begin{figure}[H]
\centering
\includegraphics[width=0.7\textwidth]{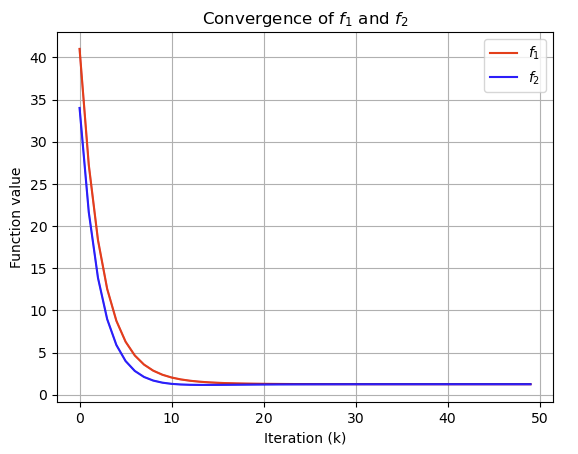}
\caption{Weighted-Sum Method}
\label{weight-sum}
\end{figure}

The last example illustrates the Gerstewitz scalarization method bringing us a nonsmooth resulting function in contrast to the weight-sum approach.

\begin{example}
Let $f_1(x):= (x-2)^2$ and $f_2(x) := (x+1)^2 +1$.
\end{example}

In this example, we choose the set $A:= \{(y_1,y_2) \in \R^2 \mid y_1 \hspace{0.2cm} \text{and}\hspace{0.2cm} y_2 \leq 0 \}$. We have $\phi_A (f_1, f_2) = \inf \{t \mid (f_1,f_2) - tk^0 \in A\}$. Then $t \geq f_1$ and $t \geq f_2$. Thus $t \geq \max \{f_1,f_2\} $ with the infimum of $t$ calculated by $t=\max \{f_1,f_2\}$. 
\begin{figure}[H]
\centering
\includegraphics[width=0.7\textwidth]{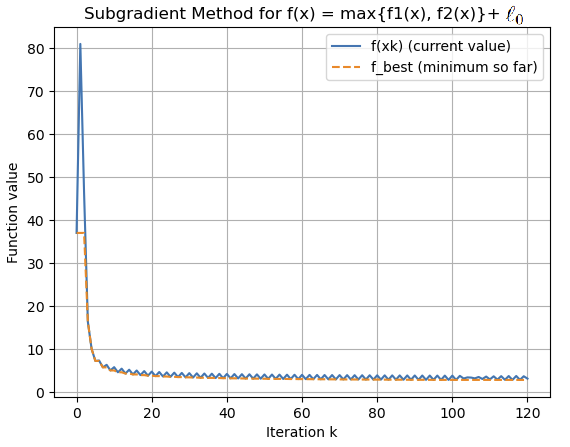}
\caption{Maximum Function}
\label{max}
\end{figure}

\section{Concluding Remarks and Future Research}\label{sec:conc}

This paper concerns single-objective and multiobjective optimization problems involving the $\ell_0$-norm function in their objectives. Problems of this type arise in models of proton therapy, which have been the main practical motivations of our research. Being intrinsically nonsmooth and nonconvex, such problems require the usage advanced tools of variational analysis for their study and applications. The main results obtained in the paper revolve around novel subgradient algorithms to solve both scalar and multiobjective versions based on the limiting subdifferential by Mordukhovich. In this way, we implement two scalarization techniques to deal with multiobjective problems: the weigh-sum approach and mainly Gerstewitz scalarization. The desired convergence properties of the designed algorithms are rigorously verified, and the performance of these algorithms are illustrated by numerical calculations for typical examples. 

While this study introduces a novel framework for $\ell_0$ optimization in proton radiation therapy, we acknowledge that the numerical examples provided are simplified relative the complexity of real clinical scenarios. These examples primarily serve to illustrate the feasibility and potential of proposed methodology. The main direction of our future research to implement the designed algorithms to solving realistic practical models of proton beam therapy as well as related models of cancer research. These tasks will definitely require some adjustments and modifications, which will bring us to new mathematical results.\\[1ex]
{\bf Acknowledgements}. Research of Xuanfeng Ding was partly supported by NIH under grant R01CA301448. Research of Boris Mordukhovich was partly supported by the US National Science Foundation under grant DMS-2204519, by the Australian Research Council under Discovery Project DP-190100555, and by Project 111 of China under grant D21024. Research of Anh Vu Nguyen was partly supported by the US National Science Foundation under grant DMS-2204519.

\bibliographystyle{unsrtnat}   
\bibliography{references}

\begin{thebibliography}{17}
\providecommand{\natexlab}[1]{#1}
\providecommand{\url}[1]{\texttt{#1}}
\expandafter\ifx\csname urlstyle\endcsname\relax
  \providecommand{\doi}[1]{doi: #1}\else
  \providecommand{\doi}{doi: \begingroup \urlstyle{rm}\Url}\fi

\bibitem[Khanh et~al.(2025)Khanh, Mordukhovich, and Phat]{khanh}
P.~D. Khanh, B.~S. Mordukhovich, and V.~T. Phat.
\newblock Coderivative-based newton methods in structured nonconvex and nonsmooth optimization.
\newblock \emph{arXiv:2403.04262v2}, 2025.

\bibitem[Jiao et~al.(2015)Jiao, Jin, and Lu]{jiao2015primal}
Y.~Jiao, B.~Jin, and X.~Lu.
\newblock A primal dual active set with continuation algorithm for the $\ell_0$-regularized optimization problem.
\newblock \emph{Appl. Comput. Harmon. Anal.}, 39:\penalty0 400--426, 2015.

\bibitem[Louizos et~al.(2018)Louizos, Welling, and Kingma]{louizos2017learning}
C.~Louizos, M.~Welling, and D.~P. Kingma.
\newblock Learning sparse neural networks through $\ell_0$ regularization.
\newblock \emph{Proc. Inter. Conf. Learn. Reprersen. 2018; arXiv:1712.01312}, 2018.

\bibitem[Lee et~al.(2022)Lee, Imrie, and van~der Schaar]{lee2022self}
C.~Lee, F.~Imrie, and M.~van~der Schaar.
\newblock Self-supervision enhanced feature selection with correlated gates.
\newblock In \emph{Proc. Inter. Conf. Learn. Represen.}, 2022.

\bibitem[Lyu et~al.(2019)Lyu, O’Connor, Niu, and Sheng]{lyu2019image}
Q.~Lyu, D.~O’Connor, T.~Niu, and K.~Sheng.
\newblock Image-domain multimaterial decomposition for dual-energy computed tomography with nonconvex sparsity regularization.
\newblock \emph{J. Med. Imag.}, 6:\penalty0 DOI: 10.1117/1.JMI.6.4.044004, 2019.

\bibitem[et~al.(2020)]{Gu}
W.~Gu et~al.
\newblock A novel energy layer optimization framework for spot-scanning proton arc therapy.
\newblock \emph{Med. Phys.}, 47:\penalty0 2072–2084, 2020.

\bibitem[L.~Zhao(2023)]{Zhao}
et~al. L.~Zhao.
\newblock The first direct method of spot sparsity optimization for proton arc therapy.
\newblock \emph{Acta Oncol. Stockh. Swed.}, 62:\penalty0 48–52, 2023.

\bibitem[Fan et~al.(2025)Fan, Zhao, Li, Qian, Dao, Hu, Zhang, Yang, Lu, Yang, et~al.]{fan2024optimizing}
Q.~Fan, L.~Zhao, X.~Li, Y.~Qian, R.~Dao, J.~Hu, S.~Zhang, K.~Yang, X.~Lu, Z.~Yang, et~al.
\newblock Optimizing spot-scanning proton arc therapy with a novel spot sparsity approach.
\newblock \emph{Med. {P}hys.}, 52:\penalty0 1789--1797, 2025.

\bibitem[Zhao et~al.(2022)Zhao, You, Liu, Lu, and Ding]{Zhao2022ADMM}
L.~Zhao, J.~You, G.~Liu, X.~Lu, and X.~Ding.
\newblock A novel simultaneous plan quality and beam delivery time sparc optimization platform using alternating direction method of multipliers (admm).
\newblock Particle Therapy Cooperative Group, 2022.

\bibitem[Zhao et~al.(2023)Zhao, You, Liu, Wuyckens, Lu, and Ding]{zhao2023first}
L.~Zhao, J.~You, G.~Liu, S.~Wuyckens, X.~Lu, and X.~Ding.
\newblock The first direct method of spot sparsity optimization for proton arc therapy.
\newblock \emph{Acta {O}ncol.}, 62:\penalty0 48--52, 2023.

\bibitem[Mordukhovich(2006)]{mordukhovich2006}
B.~S. Mordukhovich.
\newblock Variational analysis and generalized differentiation, i: Basic theory, ii: Applications.
\newblock \emph{Springer, Berlin}, 2006.

\bibitem[Rockafellar and Wets(1998)]{rockafellar1998}
R.~T. Rockafellar and R.~J-B. Wets.
\newblock Variational analysis.
\newblock \emph{Springer, Berlin}, 1998.

\bibitem[(Tammer)(1983)]{tammer1983}
C.~Gerstewicz (Tammer).
\newblock Nichtknvexe dualit\"at in der vektoroptimierung.
\newblock \emph{Wissenschaftitiche Zeitschrift der TH Leuna-Merseburg}, 25:\penalty0 357--364, 1983.

\bibitem[Khan et~al.(2015)Khan, Tammer, and Zălinescu]{tammer2015}
A.~A. Khan, C.~Tammer, and C.~Zălinescu.
\newblock Set-valued optimization. an introduction and applications.
\newblock \emph{Springer, Berlin}, 2015.

\bibitem[Mordukhovich(2018{\natexlab{a}})]{mordukhovich2018}
B.~S. Mordukhovich.
\newblock Variational analysis and applications.
\newblock \emph{Springer Nature, Cham, Switzerland}, 2018{\natexlab{a}}.

\bibitem[Mordukhovich(2018{\natexlab{b}})]{mordukhovich2024}
B.~S. Mordukhovich.
\newblock Second-order variational analysis in optimization, variational stability, and control: Theory, algorithms, applications.
\newblock \emph{Springer Nature, Cham, Switzerland}, 2018{\natexlab{b}}.

\bibitem[Tammer and Zălinescu(2010)]{tammer2014lipschitz}
C.~Tammer and C.~Zălinescu.
\newblock Lipschitz properties of the scalarization function and applications.
\newblock \emph{Optimization}, 59:\penalty0 305--319, 2010.

\end{thebibliography}
\end{document}